\documentclass{article}

\usepackage[english]{babel}

\usepackage[letterpaper,top=2cm,bottom=2cm,left=3cm,right=3cm,marginparwidth=1.75cm]{geometry}

\usepackage{amssymb,mathrsfs,amscd}
\usepackage{amsthm}
\usepackage{hyperref}
\usepackage{amsmath}
\usepackage{lineno,hyperref,color}
\usepackage{amsmath}
\usepackage{graphicx}

\title{Devaney chaos of multiple mappings from a set-valued view}
\author{Yingcui Zhao}

\begin{document}
\newtheorem{theorem}{Theorem}[section]
\newtheorem{corollary}[theorem]{Corollary}
\newtheorem{lemma}[theorem]{Lemma}
\newtheorem{proposition}[theorem]{Proposition}
\newtheorem{problem}[theorem]{Problem}
\newtheorem{maintheorem}[theorem]{Main Theorem}
\newtheorem{definition}[theorem]{Definition}
\newtheorem{remark}[theorem]{Remark}
\newtheorem{example}[theorem]{Example}
\newtheorem{claim}{Claim}[section]
\maketitle

\begin{abstract}
We introduce the definitions of periodic point, transitivity, sensitivity and Devaney chaos of multiple mappings from a set-valued perspective. We study the relation between multiple mappings and its continuous self-maps and show that multiple mappings and its continuous self-maps do not imply each other in terms of periodic point or transitivity. While the sufficient condition for multiple mappings to have periodic points, to be transitive and sensitive is provided separately. And we show that transitivity plus dense periodic point set implies sensitivity. Also, a sufficient condition for multiple mappings to be Devaney chaotic is given.
\end{abstract}

\section{Introduction}
A dynamical system typically refers to the pair $(X,f)$, in which $X$ is a compact metric space with a metric $d$ and $f$ is a continuous self-map on $X$. One of the most interesting features of dynamical systems is that the trajectories of nearby points deviate after a finite number of steps. This is also one of the important characteristics used to describe chaotic behavior in systems. This concept, commonly known as the "butterfly effect", has been extensively studied and is referred to as sensitivity to initial conditions in \cite{sen}. And in the theory of dynamical systems, chaos is an important phenomenon. 

 Let $\mathbb{N}=\{0,1,2,\cdots\}$, $\mathbb{Z^+}=\{1,2,3,\cdots\}$. The point $x\in X$ is said to be a \emph{periodic point} of $f$ if there exists $n\in\mathbb{Z^+}$ such that $f^n(x)=x$. The \emph{period of $x$} is the smallest number $n\in\mathbb{Z^+}$ satisfying $f^n(x)=x$. Specifically, if $n=1$ we say $x$ is a fixed point of $f$. We denote the set of all periodic points for $f$ by $P(f)$.	We say that the dynamical system $(X,f)$ (or the map $f$) is \emph{(topologically) transitive}, if for any nonempty open sets $U,V\subset X$, there exists $n\in\mathbb{Z^+}$ such that $f^n(U)\bigcap V\neq\emptyset$. $f$ is \emph{sensitive}, if there exists $\delta>0$ such that for any nonempty open set $U\subset X$, there exist $x,y\in U$ and $n\in \mathbb{Z^+}$ such that $d(f^n(x),f^n(y))>\delta$.
 We say that the dynamical system $(X,f)$ (or the map $f$) is \emph{Devaney chaotic}\cite{Devaney1989}, if it holds all of the following conditions:
 \begin{enumerate}
 	\item [(1)]$f$ is transitive.
 	\item[(2)]$\overline{P(f)}=X$, i. e., the set of periodic points of $f$ is dense in $X$.
 	\item[(3)]$f$ is sensitive.
 \end{enumerate}
It is worth mentioning that Banks et al. \cite{Bank1992} claimed that the above conditions (1) plus (2) implies (3).
 
 In 2016, Hou and Wang\cite{HouWang2016} defined multiple mappings derived from iterated function system. Iterated function system is an important branch of fractal theory, reflecting the essence of the world. They are one of the three frontiers of nonlinear science theory. Hou and Wang's focus was primarily on studying the Hausdorff metric entropy and Friedland entropy of multiple mappings. Additionally, they introduced the notions of Hausdorff metric Li-Yorke chaos and Hausdorff metric distributional chaos from a set-valued perspective. It is worth noting that researchers studying iterated function systems often approach the topic from a group perspective rather than a set-valued perspective. This also establishes a close connection between multiple mappings and set-valued mappings, or we can consider multiple mappings as a special case of set-valued mappings.
 
 It is important to acknowledge the valuable role of set-valued mappings in addressing complex problems involving uncertainty, ambiguity, or multiple criteria. Set-valued mappings offer versatility and flexibility, making them highly beneficial across various fields. One prominent application of set-valued mappings is in optimization problems, where the objective is to identify the optimal set of solutions. For instance, in multi-objective optimization, a set-valued mapping can represent the Pareto front, encompassing all non-dominated solutions.
 
 Set-valued mappings also prove useful in decision-making processes that require considering multiple criteria or preferences. By representing feasible solutions as sets, decision-makers can thoroughly analyze and compare different options, enabling them to make well-informed decisions. Additionally, set-valued mappings find applications in data analysis tasks such as clustering and classification. Unlike assigning each data point to a single category, set-valued mappings can represent uncertainty or ambiguity by assigning data points to multiple categories. In fact, the applications of set-valued mappings are vast and diverse, encompassing numerous fields beyond those mentioned here.

In \cite{ziji}, we studied that if multiple mappings $F$ has a disdributionally chaotic
pair, especially $F$ is distributionally chaotic, $S\Omega(F)$
contains at least two points and gives a sufficient condition for
$F$ to be distributionally chaotic in a sequence and chaotic in the
strong sense of Li-Yorke. Zeng et al. \cite{zeng2020} proved two topological conjugacy dynamical systems to multiple mappings have simultaneously Hausdorff metric Li-Yorke chaos or Hausdorff distributional chaos and the multiple mappings $F=\{f_1,f_2\}$ and its $2$-tuple of continuous self-maps $f_1,f_2$ are not mutually implied in terms of Hausdorff metric Li-Yorke chaos.

The current paper aims to consider the image of one point under multiple mappings as a set (a compact set). 
The specific layout of the present paper is
as follows. Some preliminaries and definitions are introduced in Section \ref{sec2}.
Then we study the relation between multiple mappings and its continuous self-maps in terms of periodic point, transitivity and sensitivity in Section \ref{sec3}, and the Devaney chaos of multiple mappings in Section \ref{sec4}. At last, we state the conclusions of this paper in Section \ref{sec5}.

\section{Preliminaries}\label{sec2}
Let $F=\{f_1,f_2,\cdots,f_n\}$ be a \emph{multiple mappings} with $n$-tuple of continuous self-maps on $X$. Note that for any $x\in X$, $F(x)=\{f_1(x),f_2(x),\cdots,f_n(x)\}\subset X$ is compact. Let $$\mathbb{K}(X)=\{K\subset X|K~\text{is nonempty and compact}\}.$$
Then $F$ is from $X$ to $\mathbb{K}(X)$. The Hausdorff metric $d_H$ on $\mathbb{K}(X)$ is defined by $$d_H(A,B)=\max\{\sup_{a\in A}\inf_{b\in B}d(a,b),\sup_{b\in B}\inf_{a\in A}d(a,b)\},\forall A,B\subset X.$$ It is clear that $\mathbb{K}(X)$ is a compact metric space with the Hausdorff metric $d_H$. 

For convenience, let us study multiple mappings on compact metric spaces from a set-valued view, using the example of examining two continuous self-maps. The definitions and conclusions presented in this paper can be easily extended to the case of a multiple mappings formed by any finite number of continuous self-maps. Consider the multiple mappings $F=\{f_1,f_2\}$. For any $n>0$, $F^n:X\rightarrow\mathbb{K}(X)$ is defined by for  any $x\in X$,$$F^n(x)=\{f_{i_1}f_{i_2}\cdots f_{i_n}(x)|i_1,i_2,\cdots,i_n=1\text{ or }2\}.$$
It is obvious that $F^n(x)\in \mathbb{K}(X)$. For any $A\subset X$, let $$F^n(A)=\{f_{i_1}f_{i_2}\cdots f_{i_n}(a)|a\in A, i_1,i_2,\cdots,i_n=1\text{ or }2\}=\bigcup_{a\in A}F^n(a).$$
Particularly, if $A\in\mathbb{K}(X)$, $F^n(A)\in\mathbb{K}(X)$. Then, $F$ naturally induces a continuous self-map on $\mathbb{K}(X)$, denoted by $\widetilde{F}:\mathbb{K}(X)\rightarrow\mathbb{K}(X)$.

Throughout this paper, 	Suppose that $F=\{f_1,f_2\}$ be a multiple mappings with $2$-tuple of continuous self-maps on $X$. Now, we define the concepts of periodic point, transitivity, sensitivity and Devaney chaos for multiple mappings $F=\{f_1,f_2\}$ from a set-valued view. 

\begin{definition}
	For any $x\in X$, we say $\{F(x),F^2(X),F^3(X),\cdots\}:=\mathring{orb}(x,F)$ is the deleted orbit of $x$ under $F$. 
\end{definition}

\begin{definition}\label{pfp}
	$x\in X$ is said to be a periodic point of $F$, if the deleted orbit of $x$ under $F$ $\mathring{orb}(x,F)$ is finite and there exists $m>0$ such that $x\in F^m(x)$. The cardinal number of $\mathring{orb}(x,F)$ denoted by $\sharp\mathring{orb}(x,F)$ is said to be the period (or minimum positive period) of $x$ for $F$. Especially, if $F(x)=\{x\}$, $x$ is said to be a fixed point for $F$. The set of all period points of $F$ is denoted by $P(F)$.
\end{definition}
\begin{definition}\label{newde1}
 We say $F$ is 
 	\begin{enumerate}	
 	\item[(1)] \emph{(Hausdorff metric) sensitive}, if there exists $\delta>0$ such that for any nonempty open set $U\subset X$, there exist $x,y\in U$ and $n\in \mathbb{Z^+}$ such that $d_H(F^n(x),F^n(y))>\delta$. 
\item[(2)] \emph{(topologically) transitive}, if for any nonempty open sets $U\subset X$ and $\mathcal{U}\subset Ran(F):=\{F^n(x)\mid n\geq 1,x\in X\}$, there exists $n\in\mathbb{Z^+}$ such that $$\{F^n(u)\mid u\in U\}\bigcap\mathcal{U}\neq\emptyset.$$
		
	\end{enumerate}
\end{definition}
\begin{definition}
	We say that $F$ is \emph{Devaney chaotic}, if it holds all of the following conditions:
	\begin{enumerate}
		\item [(1)]$F$ is transitive.
		\item[(2)]$\overline{P(F)}=X$, i. e., the set of periodic points of $F$ is dense in $X$.
		\item[(3)]$F$ is sensitive.
	\end{enumerate}
\end{definition}

Clearly, the periodic point, sensitivity and transitivity of multiple mappings, in the case of degradation (where the multiple mappings consists of only one continuous self-map), is the same as the periodic point, sensitivity and transitivity of a classical single continuous self-map. Next we provide an example to illustrate the existence of the newly defined concept Definition \ref{newde1}. 
\begin{example}
	Consider the multiple map defined on $[0,1]$ as $F=\{f_1,f_2\}$, in which 
	\begin{equation*}
		\begin{aligned}
				f_1(x)=\begin{cases}
						2x, & 0\leq x\leq\frac{1}{2}, \\
						2-2x, & \frac{1}{2}<x\leq 1,
					\end{cases}
		\end{aligned}
	\begin{aligned}
			f_2(x)=\begin{cases}
					1-2x, & 0\leq x\leq\frac{1}{2}, \\
					2x-1, & \frac{1}{2}<x\leq 1.
				\end{cases}
	\end{aligned}
	\end{equation*}

It is can be verified that for any $x\in [0,1]$ and any $n>0$, $F^n(x)=\{f_1^n(x),f_2^n(x)\}$ and $f_1^n(x)+f_2^n(x)=1$.  Let $U$ be nonempty open set of $[0,1]$. Then there exist $x,y\in U$ and $n\neq m$ such that $f_1^n(x)=f_2^n(x)=\frac{1}{2}$ and $f_1^m(y)=f_2^m(y)=\frac{1}{2}$. Without loss of generality, let $n>m$. Then, $F^n(x)=\{\frac{1}{2}\}$ and $F^n(y)=\{0,1\}$. Thus, $d_H(F^n(x),f^n(y))>\frac{1}{2}$. So, $F$ is sensitive.

\end{example}
\begin{example}\label{e1}
	Consider the multiple map defined on $[0,1]$ as $F=\{f_1,f_2\}$, in which $f_1(0)=f_1(1)=0$ and $f_2(0)=f_2(1)=1$. Then $Ran(F)=\{\{0,1\}\}$ and $F(0)=\{0,1\}=F(1)$. So, $F$ is transitive.
\end{example}
\section{Relation Between $F$ and $f_1,f_2$}\label{sec3}
A natural question is what is the implication between the fixed point / periodic point / transitivity / sensitivity / Devaney chaos of multiple mappings $F=\{f_1,f_2\}$ and the corresponding properties of its $2$-tuple of continuous self-maps $f_1,f_2$?
\subsection{Peiodic points}

First, according to Definition \ref{pfp}, the following proposition can be directly obtained.
\begin{proposition}
	$x$ is a fixed point of $F$ if and only if $x$ is a common fixed point of $f_1$ and $f_2$. 
\end{proposition}
While, the periodic point of $F$ may not necessarily be the period point of $f_1$ or $f_2$.
\begin{example}\label{yyzeng}
	Consider the multiple mappings defined on $[0,1]$ as $F=\{f_1,f_2\}$, in which 
		\begin{equation*}
		\begin{aligned}
		f_1(x)=\begin{cases}
	2x, & 0\leq x\leq\frac{1}{2}, \\
	1, & \frac{1}{2}<x\leq 1,
\end{cases}
		\end{aligned}
		\begin{aligned}
		f_2(x)=\begin{cases}
	1, & 0\leq x\leq\frac{1}{2}, \\
	2-2x, & \frac{1}{2}<x\leq 1.
\end{cases}
		\end{aligned}
	\end{equation*}

\begin{itemize}
	\item [(1)]Since $\mathring{orb}(\frac{2}{7},F)=\{\{\frac{4}{7},1\},\{0,\frac{6}{7},1\},\{0,\frac{2}{7},1\},\{0,\frac{4}{7},1\}\}$, then $\frac{2}{7}$ is a periodic point
	 with period $4$ of $F$. 
	\item[(2)]Since $f_1^n(\frac{2}{7})=1$, $\forall n\geq 2$, then $\frac{2}{7}$ is not a periodic point of $f_1$.
	\item[(3)]Since for any even number $n>0$, $f_2^n(\frac{2}{7})=1$ and for any odd number $m>0$, $f_2^m(\frac{2}{7})=0$, then $\frac{2}{7}$ is not a periodic point of $f_2$.
\end{itemize}
\end{example}
The converse may not necessarily hold true, that is, the common periodic point of $f_1$ and $f_2$ may not necessarily be the period point of $F$.
\begin{example}
		Consider the multiple mappings defined on $\{\frac{1}{n}\mid n=1,2,\cdots\}\bigcup\{0\}$ as $F=\{f_1,f_2\}$, in which $f_1:1\rightarrow \frac{1}{2} \rightarrow \frac{1}{3}\rightarrow 1$, $f_1(\frac{1}{n})=\frac{1}{n+1}$, $n=4,5,\cdots$ and $f_2:1\rightarrow \frac{1}{4} \rightarrow \frac{1}{5}\rightarrow 1$, $f_2(\frac{1}{3})=\frac{1}{6}$, $f_2(\frac{1}{n})=\frac{1}{n+1}$, $n=2,6,7,8,\cdots$
	\begin{itemize}
		\item [(1)]It is easy to see that $1$ is a periodic point with period $3$ of both $f_1$ and $f_2$.
		\item[(2)]Since $F: 1\rightarrow \{\frac{1}{2},\frac{1}{4}\}\rightarrow \{\frac{1}{3},\frac{1}{5}\}\rightarrow \{\frac{1}{6},1\}\rightarrow \{\frac{1}{2},\frac{1}{4},\frac{1}{7}\}\rightarrow\cdots$, $1$ is not a periodic point of $F$.
	\end{itemize}
\end{example}

Although the common periodic point of $f_1$ and $f_2$ may not necessarily be the period point of $F$,  $f_1f_2=f_2f_1$ makes it true. Furthermore, we also provide an additional sufficient condition for some point to be a periodic point of multiple mappings.
\begin{theorem}
	If $f_1f_2=f_2f_1$ and there exists $p\in X$ such that $f_1^n(p)=p$ for some $n>0$ and $f_2^m(p)=p$ for some $m>0$, then $p$ is a periodic point of $F$.
\end{theorem}
\begin{proof}
	Since 
	$$\sharp\mathring{orb}(x,F)\leq C^1_{(n-1)(m-1)}C^2_{(n-1)(m-1)}\cdots C^{(n-1)(m-1)}_{(n-1)(m-1)}(n-1)(m-1),$$ $\mathring{orb}(x,F)$ is finite. And $p\in F^n(x)$. So, $p$ is a periodic point of $F$.
\end{proof}
\begin{theorem}\label{ppp}
	If $f_1(x)=c$ ($\forall x\in X$, $c$ is a constant), $f_2(c)=c$ and $p\in X$ is a periodic point of $f_2$, then $p$ is a periodic point of $F$.
\end{theorem}
\begin{proof}
	It is esay to see that $Ran(F)=\{\{c,f_2^n(x)\}\mid  n\geq 1, x\in X\}$. Let $p\in X$ be a periodic point of $f_2$ with a period of $m>0$, then  $$\mathring{orb}(x,F)=\{\{c,p\},\{c,f_2(p)\},\{c,f_2^2(p)\},\cdots,\{c,f_2^{m-1}(p)\}\}$$
	and $p\in \{c,f_2^{m}(p)\}=F^m(p)$. So, $p$ is a periodic point of $F$.

\end{proof}
\subsection{Transitivity}
The Example \ref{e1} shows that the transitivity of multiple mappings $F=\{f_1,f_2\}$ doesn't imply the transitivity of $f_1$ or $f_2$. Then combined with the following example, the transitivity of multiple mappings $F=\{f_1,f_2\}$ and its $2$-tuple of continuous self-maps $f_1,f_2$ do not imply each other.

\begin{example}
	Consider the multiple mappings  $F=\{f_1,f_2\}$ defined on $\{0,1,2\}$, in which $f_1:0\longmapsto1\longmapsto2\longmapsto0$ and $f_2:0\longmapsto2\longmapsto1\longmapsto0$. It is easy to see that both $f_1$ and $f_2$ is transitive. Next we show $F$ is not transitive.
	
	$Ran(F)=\{\{1,2\},\{0,2\},\{0,1\},\{0,1,2\}\}$. Let $U=\{0\}$ and $\mathcal{U}=\{\{0,2\}\}$. Then $U$ is a nonempty open set of $\{0,1,2\}$ and $\mathcal{U}$ is a nonempty open set of $Ran(F)$. While $F:0\mapsto\{1,2\}\mapsto\{0,1,2\}\mapsto\cdots\mapsto\{0,1,2\}\mapsto\cdots$. So, $F$ is not transitive.
\end{example}

Although that both $f_1$ and $f_2$ are transitive can't imply $F=\{f_1,f_2\}$ is transitive, we give a sufficient condition for $F$ to be transitive as Theorem \ref{thFtra}.
\begin{theorem}\label{thFtra}
	If $f_1(x)=c$ ($\forall x\in X$, $c$ is a constant), $f_2(c)=c$ and $f_2$ is transitive, then $F$ is transitive.
\end{theorem}
\begin{proof}
	It is esay to see that $Ran(F)=\{\{c,f_2^n(x)\}\mid n\geq 1, x\in X\}$. Let $U\subset X$ and $\mathcal{U}\subset Ran(F)$ be two nonempty open sets. Then there exists nonempty open set $V\subset X$ such that $\{\{c,v\}\mid v\in V\}\subset\mathcal{U}$. Since $f_2$ is transitive, there exists $n\in\mathbb{Z^+}$ such that $f_2^n(U)\bigcap V\neq\emptyset$. Then there exists $u\in U$ such that $f_2^n(u)\in V$, that is, $F^n(u)\in\mathcal{U}$. So, $F$ is transitive.
	
\end{proof}
\subsection{Sensitivity}
Firstly we show the sensitivity of $F$ can't imply that $f_1$ or $f_2$ is sensitive.
\begin{example}\label{yyzeng}
	Consider the multiple mappings defined on $[0,1]$ as $F=\{f_1,f_2\}$, in which
			\begin{equation*}
		\begin{aligned}
		f_1(x)=\begin{cases}
	2x, & 0\leq x\leq\frac{1}{2}, \\
	1, & \frac{1}{2}<x\leq 1,
\end{cases}
		\end{aligned}
		\begin{aligned}
		f_2(x)=\begin{cases}
	1, & 0\leq x\leq\frac{1}{2}, \\
	2-2x, & \frac{1}{2}<x\leq 1.
\end{cases}
		\end{aligned}
	\end{equation*}
Let
	\begin{equation*}
		f(x)=\begin{cases}
			2x, & 0\leq x\leq\frac{1}{2}, \\
			2-2x, & \frac{1}{2}<x\leq 1.
		\end{cases}
	\end{equation*}
	Then for any $x\in [0,1]$ and any $n\geq 2$, $F^n(x)=\{0,1,f^n(x)\}$. 
	\begin{itemize}
		\item[(1)]As we all know, $f$ is sensitive. Let $\delta>0$ be the sensitive constant for $f$. Then for any nonempty open set $U\subset [0,1]$, there exist $x,y\in U$ and $n>0$ such that $d(f^n(x),f^n(y))>\delta$. Therefore, there exist $x,y\in U$ and $n>0$ such that $d_H(F^n(x),F^n(y))>\delta$. So, $F$ is sensitive.
		\item[(2)]Let $U=(\frac{1}{2},1)$. Then for any $x,y\in U$ and any $n>0$, $f_1^n(x)=f_1^n(y)=1$. So, $f_1$ is not sensitive.
		\item[(3)]Let $U=(0,\frac{1}{2})$. Then for any $x,y\in U$ and any $n>0$, $f_2^n(x)=f_2^n(y)=1$ or $0$. So, $f_2$ is not sensitive.
	\end{itemize}
\end{example}
Considering this implication in reverse, we give a sufficient condition for multiple mappings to be sensitive.
\begin{theorem}\label{thFsen}
	If there exists $\lambda>1$ such that for any nonempty open sets $U,V\subset X$, there exist $i_0=1$ or $2$ such that there exists $x\in U$ and $y\in V$ satisfying $$d(f_{i_0}(x),f_j(y))>\lambda d(x,y), \forall j=1,2,$$
	then $F$ is sensitive.
\end{theorem}
\begin{proof}
	Let $U$ be a nonempty open set of $X$. Select $x_0\neq y_0\in U$ satisfying $$d(f_{\alpha_1}(x_0),f_{\beta_1}(y_0))>\lambda d(x_0,y_0),\forall \beta_1=1,2,\exists\alpha_1=1,2.$$ Select $n$ with $\frac{\lambda^{n+1}d(x_0,y_0)}{n+1}>\frac{diam X}{2}$.
	
	\begin{itemize}
		\item [Step $1$:] Take $\eta_1=\frac{\lambda^2 d(x_0,y_0)}{8}>0$. By $f_1$ and $f_2$ are continuous and $X$ is compact, there exists $0<\delta_1<\frac{\lambda d(x_0,y_0)}{8}$ such that for any $x,y\in X$, $$d(x,y)<\delta_1\Rightarrow d(f_i(x),f_i(y))<\eta_1,\forall i=1,2.$$
		
		Then, there exist $u_{\alpha_1}\in B(f_{\alpha_1}(x_0),\delta_1)$ and $v_{\beta_1}\in B(f_{\beta_1}(y_0),\delta_1)$ such that 
		$$d(f_{\alpha_2}(u_{\alpha_1}),f_{\beta_2}(v_{\beta_1}))>\lambda d(u_{\alpha_1},v_{\beta_1}),\forall \beta_2=1,2,\exists\alpha_2=1,2.$$
		
		And $$d(u_{\alpha_1},v_{\beta_1})>d(f_{\alpha_1}(x_0),f_{\beta_1}(y_0))-d(f_{\alpha_1}(x_0),u_{\alpha_1})-d(v_{\beta_1},f_{\beta_1}(y_0))>\lambda d(x_0,y_0)-2\delta_1.$$ Then, $$d(f_{\alpha_2}(u_{\alpha_1}),f_{\beta_2}(v_{\beta_1}))>\lambda^2 d(x_0,y_0)-2\delta_1\lambda.$$ Thus,
		\begin{align}
			&d(f_{\alpha_2}f_{\alpha_1}(x_0),f_{\beta_2}f_{\beta_1}(y_0)) \nonumber \\   
			>&d(f_{\alpha_2}(u_{\alpha_1}),f_{\beta_2}(v_{\beta_1}))-d(f_{\alpha_2}f_{\alpha_1}(x_0),f_{\alpha_2}(u_{\alpha_1}))-d(f_{\beta_2}(v_{\beta_1}),f_{\beta_2}f_{\beta_1}(y_0))\nonumber\\
			>&\lambda^2 d(x_0,y_0)-2\delta_1\lambda-2\eta_1\nonumber\\
			>&\lambda^2 d(x_0,y_0)-2\lambda\frac{\lambda d(x_0,y_0)}{8}-2\lambda\frac{\lambda^2 d(x_0,y_0)}{8}\nonumber\\
			=&\frac{\lambda^2}{2}d(x_0,y_0).\nonumber
		\end{align}
		\item [Step $2$:] Take $\eta_2=\frac{\lambda^3 d(x_0,y_0)}{24}>0$. By $f_1$ and $f_2$ are continuous and $X$ is compact, there exists $0<\delta_2<\frac{\lambda^2 d(x_0,y_0)}{24}$ such that for any $x,y\in X$, $$d(x,y)<\delta_2\Rightarrow d(f_i(x),f_i(y))<\eta_2,\forall i=1,2.$$
		
		Then, there exist $u_{\alpha_2}\in B(f_{\alpha_2}f_{\alpha_1}(x_0),\delta_2)$ and $v_{\beta_2}\in B(f_{\beta_2}f_{\beta_1}(y_0),\delta_2)$ such that 
		$$d(f_{\alpha_3}(u_{\alpha_2}),f_{\beta_3}(v_{\beta_2}))>\lambda d(u_{\alpha_2},v_{\beta_2}),\forall \beta_3=1,2,\exists\alpha_3=1,2.$$
		
		And
		\begin{align}
			&d(u_{\alpha_2},v_{\beta_2}) \nonumber \\   
			>&d(f_{\alpha_2}f_{\alpha_1}(x_0),f_{\beta_2}f_{\beta_1}(y_0))-d(f_{\alpha_2}f_{\alpha_1}(x_0),u_{\alpha_2})-d(v_{\beta_2},f_{\beta_2}f_{\beta_1}(y_0))\nonumber\\
			>&\frac{\lambda^2}{2} d(x_0,y_0)-2\delta_2.\nonumber
		\end{align}
		
		Then,
		\begin{align}
			&d(f_{\alpha_3}f_{\alpha_2}f_{\alpha_1}(x_0),f_{\beta_3}f_{\beta_2}f_{\beta_1}(y_0)) \nonumber \\   
			>&d(f_{\alpha_3}(u_{\alpha_2}),f_{\beta_3}(v_{\beta_2}))-d(f_{\alpha_3}f_{\alpha_2}f_{\alpha_1}(x_0),f_{\alpha_3}(u_{\alpha_2}))-d(f_{\beta_3}(v_{\beta_2}),f_{\beta_3}f_{\beta_2}f_{\beta_1}(y_0))\nonumber\\
			>&\frac{\lambda^3}{2} d(x_0,y_0)-2\delta_2\lambda-2\eta_2 \nonumber\\
			>&\frac{\lambda^3}{2} d(x_0,y_0)-2\lambda\frac{\lambda^2 d(x_0,y_0)}{24}-2\frac{\lambda^3 d(x_0,y_0)}{24}\nonumber\\
			=&\frac{\lambda^3}{3}d(x_0,y_0).\nonumber
		\end{align}			
		\item[$\cdots$ $\cdots$]
		
		\item[Step $n$:] The same goes for this step $n$. $\eta_n=\frac{\lambda^{n+1} d(x_0,y_0)}{4n(n+1)}>0$. By $f_1$ and $f_2$ are continuous and $X$ is compact, there exists $0<\delta_n<\frac{\lambda^n d(x_0,y_0)}{4n(n+1)}$ such that for any $x,y\in X$, $$d(x,y)<\delta_n\Rightarrow d(f_i(x),f_i(y))<\eta_n,\forall i=1,2.$$
		
		Then, there exist $u_{\alpha_n}\in B(f_{\alpha_{n}}\cdots f_{\alpha_1}(x_0),\delta_n)$ and $v_{\beta_n}\in B(f_{\beta_{n}}\cdots f_{\beta_1}(y_0),\delta_n)$ such that 
		$$d(f_{\alpha_{n+1}}(u_{\alpha_n}),f_{\beta_{n+1}}(v_{\beta_n}))>\lambda d(u_{\alpha_n},v_{\beta_n}),\forall \beta_n=1,2,\exists\alpha_n=1,2.$$

		And
		\begin{align}
			&d(u_{\alpha_n},v_{\beta_n}) \nonumber \\   
			>&d(f_{\alpha_{n}}\cdots f_{\alpha_1}(x_0),f_{\beta_{n}}\cdots f_{\beta_1}(y_0))-d((f_{\alpha_{n}}\cdots f_{\alpha_1}(x_0),u_{\alpha_n})-d(v_{\beta_2},f_{\beta_{n}}\cdots f_{\beta_1}(y_0))\nonumber\\
			>&\frac{\lambda^n}{n} d(x_0,y_0)-2\delta_n.\nonumber
		\end{align}
		
		Then,
		\begin{align}
			&d(f_{\alpha_{n+1}}\cdots f_{\alpha_1}(x_0),f_{\beta_{n+1}}\cdots f_{\beta_1}(y_0)) \nonumber \\   
			>&d(f_{\alpha_{n+1}}(u_{\alpha_n}),f_{\beta_{n+1}}(v_{\beta_n}))-d(f_{\alpha_n}\cdots f_{\alpha_1}(x_0),f_{\alpha_{n+1}}(u_{\alpha_n}))-d(f_{\beta_{n+1}}(v_{\beta_n}),f_{\beta_{n+1}}\cdots f_{\beta_1}(y_0))\nonumber\\
			>&\frac{\lambda^{n+1}}{n} d(x_0,y_0)-2\delta_n\lambda-2\eta_n \nonumber\\
			>&\frac{\lambda^{n+1}}{n} d(x_0,y_0)-2\frac{\lambda^{n+1}}{4n(n+1)} d(x_0,y_0)-2\frac{\lambda^{n+1}}{4n(n+1)}d(x_0,y_0) \nonumber\\
			=&\frac{\lambda^{n+1}}{n+1}d(x_0,y_0)>\frac{diam X}{2}.\nonumber
		\end{align}		
	\end{itemize}	
	
	So, $F$ is sensitive.
\end{proof}
Now let's illustrate Theorem \ref{thFsen}.
\begin{example}
	Let $S^1$ be the unit circle on the complex plane. Define the metric on $S^1$ by $d(e^{i\alpha},e^{i\beta})=\frac{\mid \alpha-\beta\mid }{2\pi}$ for any $e^{i\alpha},e^{i\beta}\in S^1$. Consider the multiple mappings defined on $S^1$ as $F=\{f_1,f_2\}$, in which $f_1(e^{i\alpha})=e^{i2\alpha}$ and $f_2(e^{i\alpha})=e^{i3\alpha}$. 
	
	Let $U,V\subset X$ be nonempty open sets. Without loss of generality, there exists $i_0=2$ such that there exist $e^{i\alpha}\in U$ and $e^{i\beta}\in V$ with $\alpha>\beta$ satisfy $$d(f_2(e^{i\alpha}),f_2(e^{i\beta}))=d(e^{i3\alpha},e^{i3\beta})=\frac{(3\alpha-3\beta)}{2\pi}=3d(e^{i\alpha},e^{i\beta})>\frac{3}{2}d(e^{i\alpha},e^{i\beta}),$$ 
	and	
	$$d(f_2(e^{i\alpha}),f_1(e^{i\beta}))=d(e^{i3\alpha},e^{i2\beta})=\frac{(3\alpha-2\beta)}{2\pi}>\frac{(2\alpha-2\beta)}{2\pi}>\frac{3}{2}d(e^{i\alpha},e^{i\beta}).$$
	
	So, $F$ is sensitive.
\end{example}

\section{Devaney chaos}\label{sec4}
Banks et al. \cite{Bank1992} proved that if $(X,f)$ is transitive and its periodic point set is dense in $X$ then $f$ is sensitive. Now we prove a similar conclusion for multiple mappings, but it is important to note that the proof process is not identical.
\begin{theorem}\label{dc}
	If $F$ is transitive and $\overline{P(F)}=X$, then $F$ is sensitive.
\end{theorem}
\begin{proof}
	Firstly, it is claimed that under the assumption of the theorem, there exists $\delta_0>0$ such that for any $x\in X$ there exists $p\in P(F)$ satisfying $$d(x,\mathring{orb}(p,F))\geq \frac{\delta_0}{2}.$$ In fact, for any period points $p_1$ and $p_2$ of $F$ with distinct deleted orbits, let $$\delta_0=d(\mathring{orb}(p_1,F),\mathring{orb}(p_2,F)).$$ Since $d(x,\mathring{orb}(p_1,F))<\frac{\delta_0}{2}$ and $d(x,\mathring{orb}(p_2,F))<\frac{\delta_0}{2}$ cannot hold simultaneously, then $\delta_0>0$ satisfies the requirement. 
	
	Let $\delta=\frac{\delta_0}{8}$. Now we show $\delta$ is a sensitive constant of $F$. For any given $x\in X$ and any given $0<\epsilon<\delta$, by $\overline{P(F)}=X$, there exists $p\in X$ such that $d(p,x)<\epsilon$, $p\in F^m(p)$ and $F^m(p)=F^{jn+m}(p)$ for some $n>0$, some $m>0$ and any $j>0$. Then 
	\begin{equation}\label{eq10}
		d(x,F^m(p))\leq d(x,p)<\epsilon<\delta.
	\end{equation}
	By the above claim, there exists $q\in P(F)$ such that 
	\begin{equation}\label{eq20}
		d_H(\{x\},\mathring{orb}(q,F))\geq 4\delta.
	\end{equation}
	Let $$\mathcal{U}=\bigcap_{i=m}^{n+m}\widetilde{F}^{-i}B_{d_H}(F^i(q),\delta).$$ Then $\mathcal{U}$ is a nonempty open set of $Ran(F)$. Since $F$ is transitive, there exist $y\in B_d(x,\epsilon)$ and $k>0$ such that $F^k(y)\in\mathcal{U}$. Let $j=[\frac{k}{n}]+1$. Then $k\leq jn\leq n+k$, $0\leq jn-k\leq n$, $m\leq jn-k+m\leq n+m$. Thus,
	\begin{equation}\label{eq30}
			 F^{jn+m}(y)=\widetilde{F}^{jn+m-k}(F^k(y))\in\widetilde{F}^{jn+m-k}(\mathcal{U})\subset B_{d_H}(F^{jn+m-k}(q),\delta).
	\end{equation}
By Equation (\ref{eq10}), Equation (\ref{eq20}) and Equation (\ref{eq30}), 
	\begin{align}
	&d_H(F^{jn+m}(p),F^{jn+m(y)}) \nonumber \\   
	\geq&d(x,F^{jn+m-k}(q))-d_H(F^{jn+m}(y),F^{jn+m-k}(q))-d(x,F^m(p))\nonumber\\
	>&4\delta-\delta-\delta\nonumber\\
	=&2\delta.\nonumber
\end{align}
Then by  triangle inequality, at least one of $d_H(F^{jn+m}(p),F^{jn+m(x)})>\delta$ and $d_H(F^{jn+m}(y),F^{jn+m(x)})>\delta$ holds true. And $y,p\in B_d(x,\epsilon)$, so $F$ is sensitive.
\end{proof}
Based on Theorem \ref{ppp}, Theorem \ref{thFtra} and Theorem \ref{dc}, the following corollary can be derived. 
\begin{corollary}
		If $f_1(x)=c$ ($\forall x\in X$, $c$ is a constant), $f_2(c)=c$ and $f_2$ is Devaney chaotic, then $F$ is Devaney chaotic.
\end{corollary}

\section{Conclusions}\label{sec5}

This paper set out to study periodic points, transitivity, sensitivity and Devaney chaos of multiple mappings from the perspective of a set-valued view. This perspective is different from the semigroup-theoretic view that has been previously studied in the context of dynamical systems of iterated function systems. We discussed the relationship between multiple mappings and its continuous self-maps in terms of the above topological properties. It may not only deepen our understanding of continuous self-maps but also enable us to use relatively simple continuous self-maps to comprehend relatively complex multiple mappings. We show that for multiple mappings, transitivity $+$ $\overline{P(F)}=X$ implies sensitivity, then a sufficient condition for multiple mappings to be Devaney chaotic is given. We hope that these conclusions can enrich the achievements in the field of multiple mappings and provide some assistance for further in-depth research in this area.

\end{document}